\documentclass[a4paper, 12pt]{article}
\usepackage[utf8]{inputenc}
\usepackage[top=1.6in, bottom=1.6in, left=1in, right=1in]{geometry}
\usepackage{amsmath, amssymb, amsthm}
\usepackage{mathrsfs}
\usepackage{enumerate}
\usepackage{hyperref}
\usepackage{color} 
\usepackage{float}
\usepackage{tikz}

\newcommand{\nn}{\mathbb{N}}
\newcommand{\zn}{\mathbb{Z}}

\newcommand{\cn}{\mathbb{C}}
\newcommand{\fk}{\mathbb{K}}

\newcommand{\hs}{\mathcal{H}}

\newcommand{\ns}{\mathcal{X}}
\newcommand{\bb}{\boldsymbol{b}}
\newcommand{\be}{\boldsymbol{e}}
\newcommand{\bh}{\boldsymbol{h}}

\newcommand{\bv}{\boldsymbol{v}}
\newcommand{\bx}{\boldsymbol{x}}
\newcommand{\by}{\boldsymbol{y}}
\newcommand{\bz}{\boldsymbol{z}}
\newcommand{\ze}{\boldsymbol{0}}
\newcommand{\norm}[1]{\left\lVert #1 \right\rVert}

\newtheorem{definition}{Definition}[section]
\newtheorem{proposition}{Proposition}[section]
\newtheorem{corollary}[proposition]{Corollary}
\newtheorem{theorem}[proposition]{Theorem}
\newtheorem{lemma}[proposition]{Lemma}
\newtheorem{remark}[proposition]{Remark}

\theoremstyle{definition}
\newtheorem{example}{Example}[section]

\numberwithin{equation}{section}

\DeclareMathOperator{\spn}{span}

\title{{\bf Asymptotic Expansive and Hyperbolic Linear Operators}}   
\author{Priyabrata Bag$^a$, Pabitra Narayan Mandal$^b$, Pramod Kumar Das$^{a,*}$\\
\footnotesize $^a$ Department of Mathematics and Statistics, School of Science,\\
\footnotesize Gandhi Institute of Technology and Management (Deemed to be University),\\
\footnotesize Doddaballapur Taluk, Bengaluru, Karnataka 561203, India\\
\footnotesize $^b$ Centre for Data Science, Institute of Technical Education and Research,\\
\footnotesize Siksha `O’ Anusandhan (Deemed to be University), Bhubaneswar 751030, India\\
\footnotesize
priyabrata.bag@gmail.com; pabitranarayanm@gmail.com; pramod.math.ju@gmail.com}   
\date{}

\begin{document}
\maketitle

\begin{abstract}
\noindent We study the topological dynamical notion of asymptotic expansivity for linear operators. We also introduce two new notions, namely strong asymptotic expansivity and super asymptotic expansivity. We show that all three notions are equivalent to hyperbolicity for linear operators on finite dimensional spaces. In case of infinite-dimension, we construct examples to show that asymptotic expansivity is weaker than strong asymptotic expansivity which is further weaker than super asymptotic expansivity. More interestingly, we prove that super asymptotic expansivity is equivalent to the notion of hyperbolicity. Then, we prove that the super asymptotically expansive linear operators is a dense class in strong operator topology. Finally, we show elegant ways to construct variants of asymptotically expansive linear operators.       
\mbox{}\\

\smallskip

\end{abstract} 

\section{Introduction} 

The notion of expansivity has been extensively studied in topological dynamics (see \cite{AH}) since its inception (see \cite{U}) in 1950. During the early exploration of this notion in the     

------------------------------

$^*$Corresponding author 

E-mail address: pramod.math.ju@gmail.com (P. K. Das)

\noindent area of nonlinear dynamics, Eisenberg and Hedlund have studied this notion for linear operators on Banach spaces (see \cite{E}, \cite{EH}, \cite{H}). Since the notion of expansivity played important role in understanding stable and chaotic dynamical systems, many researchers have introduced and investigated several variants such as positive expansivity (see \cite{AH}), measure expansivity (see \cite{M}) in nonlinear dynamics. Almost half a century after publication of Eisenberg and Hedlund's work (see \cite{E}, \cite{EH}, \cite{H}), recently, many authors again have focused on exploring the behavior of expansive, positively expansive, measure expansive linear operators on Banach spaces (see \cite{BCDMP}, \cite{BDOP}, \cite{LM}, \cite{M2000}).      
\medskip

In \cite{E}, Eisenberg has proved that an invertible linear operator on $\mathbb{C}^n$ is expansive if and only if it is hyperbolic, that is, the spectrum does not intersect the unit circle. In \cite{BCDMP}, Bernardes Jr. et al. have proved that an invertible linear operator on a Banach space is uniformly positively expansive if and only if the approximate point spectrum does not intersect the closed unit disk. In \cite{BDOP}, P. Bag et al. have proved that any linear operator on finite dimensional Hilbert space having an eigenvalue not lying on the unit circle is measure-expansive. 
\medskip

Recently, a new variant of expansivity called asymptotic expansivity has been introduced and investigated thoroughly by Lee et al. for continuous maps on compact metric spaces (see \cite{LMV}). In \cite{M2025}, Morales have introduced asymptotic expansivity for linear operators and have established connections with expansivity, uniform expansivity, and hyperbolicity. To the best of our knowledge, a thorough investigation of this notion for linear operators is due since then. Our purpose here is to present a comprehensive study on asymptotically expansive linear operators and its variants. 
\medskip

Another important class of linear operators are hyperbolic ones (see \cite{B},\cite{LM2022},\cite{CGP},\cite{M2000}). It is interesting to note that the class of hyperbolic linear operators is open and small perturbations of such operators does not change the dynamics (see \cite{H},\cite{P},\cite{P1969}).  
\medskip

We prove that a hyperbolic operator is asymptotically expansive but the converse is not true. This encourages us to introduce a stronger variant called strong asymptotic expansivity which is weaker than hyperbolicity. This motivates us in strengthening the variant further and introduce super asymptotic expansivity which is equivalent to hyperbolicity. Interestingly, these three types of asymptotic expansivity are proven to be equivalent to hyperbolicity for linear operators on finite dimensional Banach spaces. On the other hand, the class of positively expansive linear operators intersects the class of super asymptotically expansive linear operators and is strictly smaller than the class of strong asymptotically expansive linear operators. 
\medskip

It is interesting to note that the class of super asymptotically expansive linear operators on complex Banach space is dense in the set of all bounded linear operators with respect to the strong operator topology.    
\medskip

Furthermore, for a normal operator, each of the notions of asymptotic expansivity, strong asymptotic expansivity, and super asymptotic expansivity is equivalent to the corresponding notion of its adjoint operator. These notions are invariant under linear conjugacy and taking power of the operator. We also show elegant ways to construct variants of asymptotically expansive linear operators.       
\section{Definitions and Preliminaries} 

Let $\mathbb{R}$ and $\mathbb{C}$ denote the real line and the complex plane, respectively. Let $\mathbb{K}$ stands for $\mathbb{R}$ or $\mathbb{C}$. The unit circle $\lbrace z\in \mathbb{K}:|z|=1\rbrace$ and the unit disc $\lbrace z\in\mathbb{K}:|z|<1\rbrace$ are denoted by $\mathbb{T}$ and $\mathbb{D}$, respectively. The set of all non-negative integers and the set of all positive integers are denoted by $\mathbb{N}$ and $\mathbb{Z}^+$, respectively. 
\smallskip

Let $\big(\ns,\norm\cdot\big)$ be a normed linear space over $\mathbb{K}$, and let $\mathcal{S}_{\mathcal{X}}=\lbrace \bx\in \mathcal{X}:\norm{\bx}=1\rbrace$ denote the unit sphere of $\mathcal{X}$. Let $BL(\ns)$ denote the set of all bounded linear operators on $\ns$, equipped with the operator norm 
\[\Vert T\Vert=\sup\left\lbrace \frac{\Vert T(\bx)\Vert}{\Vert \bx\Vert}:\bx\in \ns\setminus \lbrace\ze\rbrace\right\rbrace.\]

We further denote the set of all surjective bounded linear operators on $\mathcal{X}$ by $BL^s(\mathcal{X})$.   
A sequence $(T_n)_{n\in\mathbb{N}}$ in $BL(\mathcal{X})$ is said to be convergent in norm topology to $T\in BL(\mathcal{X})$ if and only if for every $\varepsilon>0$ there is $N=N(\varepsilon)$ such that $\norm{T_n-T}<\varepsilon$ for all $n\geq N$. 
\smallskip

A sequence $(T_n)_{n\in\mathbb{Z^+}}$ in $BL(\mathcal{X})$ is said to be SOT convergent to $T\in BL(\mathcal{X})$ if and only if for every $\bx\in\mathcal{X}$ and every $\varepsilon>0$ there exists $N=N(\bx,\varepsilon)$ such that $\norm{T_n(\bx)-T(\bx)}<\varepsilon$ for all $n\geq \mathbb{Z^+}$. A subset $\mathcal{Y}\subset BL(\mathcal{X})$ is called SOT dense in $BL(\mathcal{X})$ if for every $T\in BL(\mathcal{X})$, there exists a sequence $(T_n)_{n\in\mathbb{Z^+}}$ in $\mathcal{Y}$ such that $(T_n)_{n\in\mathbb{Z^+}}$ is SOT convergent to $T$.      
\smallskip

Let $\sigma(T)$, $\sigma_p(T)$ and $r(T)$ denote the spectrum, the point spectrum and the spectral radius of $T$, respectively. An operator is said to be of finite-rank if its range is finite dimensional subspace. 
\smallskip

An operator $T$ on a Hilbert space is said to be normal if it commutes with its adjoint operator i.e, if $TT^*=T^*T$.   

\begin{proposition} \cite[Theorem B.9]{EP}
\label{RDT}
Let $\mathcal{B}$ be a complex Banach space and let \(T\in BL(\mathcal{B})\), and \(\sigma(T)=\sigma_1\cup\cdots\cup\sigma_N\), where \(\sigma_i\) are closed and pairwise disjoint. Then
\[\mathcal{B}=\mathcal{B}_1\oplus\cdots\oplus \mathcal{B}_N,\]
where each \(\mathcal{B}_i\) is closed \(T\)-invariant subspace of \(\mathcal{B}\) and
\[\sigma\left(T|_{\mathcal{B}_i}\right)=\sigma_i\mbox{ for }i=1,2,\ldots,N.\]
\end{proposition}

\begin{definition}
Let $\mathcal{X}$ be a normed space. Then, $T\in BL(\ns)$ is said to be hyperbolic (see \cite{BM}, \cite{CGP}) if $\sigma(T)\subseteq \mathbb{T}^c$.  
\end{definition} 

\begin{definition}
Let $(\mathcal{X},\norm{\cdot})$ be a normed space. Then, $T\in BL(\ns)$ is said to be positively expansive (see \cite{BCDMP}) if there exists $\delta>0$ such that for $\bz\in \mathcal{X}$, 
\begin{center}
whenever $\|T^n \bz\| \le \delta \quad \text{for all } n \in \mathbb{N}, \text{ we have }
\bz=0$.          
\end{center}

Such a $\delta$ is called a positive expansivity constant of $T$.    
\end{definition}  

\section{Variants of Asymptotic Expansivity} 

In this section, we introduce asymptotic expansivity, strong asymptotic expansivity and super asymptotic expansivity for linear operators on normed spaces and establish their inter-relationships. We also establish their relations with an important class of linear operators called hyperbolic operators. For a detailed understanding of hyperbolic operators and their variants, one may refer to \cite{CGP}. 
\medskip 

We first give a useful characterization of positive expansivity which was first stated (without proof) in \cite[Proposition 19]{BCDMP} on Banach spaces. However, we give a proof that shows that the completeness of the normed space is not necessary.     

\begin{proposition} \label{positive_expansive}
For a normed space $\mathcal{X}$, let $T\in BL(\mathcal{X})$. Then, $T$ is positively expansive if and only if $\sup\limits_{n\in\mathbb{N}} \norm{T^n(\bx)}=\infty$ for all $\bx\in \mathcal{S}_ \mathcal{X}$.
\end{proposition}  

\begin{proof}
Suppose that $T$ is positively expansive with positive expansivity constant $\delta$. Assume, for a contradiction, that $\bx\in \mathcal{S}_{\mathcal{X}}$ satisfies $\|T^n (\bx)\|\leq M$ for all $n\in\mathbb{N}$, where $M>0$ is a real number. Then, $\|T^n(\frac{\delta}{M}\bx)\|\leq \delta$, which implies that $\frac{\delta}{M}\bx=\ze$. Thus, $\bx=\ze$, which is a contradiction. Therefore, $\sup\limits_{n\in\mathbb{N}} \norm{T^n(\bx)}=\infty$ for all $\bx\in \mathcal{S}_\mathcal{X}$. 
\medskip

Conversely, suppose that $\sup\limits_{n\in\mathbb{N}} \norm{T^n(\bx)}=\infty$ for all $\bx\in \mathcal{S}_\mathcal{X}$. Choose $\delta>0$ arbitrarily. Let $\by \in \mathcal{X}\setminus \lbrace \ze \rbrace$ be such that $\|T^n (\by)\|\leq \delta$ for all $n\in\mathbb{N}$. Then, $\frac{\by}{\norm{\by}}\in \mathcal{S}_{\mathcal{X}}$ with $\norm{T^n\left(\frac{\by}{\norm{\by}}\right)}\leq \frac{\delta}{\norm{\by}}$ for all $n\in\mathbb{N}$. This leads to a contradiction. Therefore, $T$ must be positively expansive.   
\end{proof}

The notion of asymptotic expansivity for linear operators was introduced in \cite{M2025}. We give examples to show that it is weaker than the notion of positive expansivity.    

\begin{definition} 
An operator $T\in BL(\mathcal{X})$ is said to be asymptotically expansive if there exists $\delta>0$ such that for $\bz \in \mathcal{X}$,
\begin{center}
whenever $\|T^n (\bz)\| \le \delta \quad \text{for all } n \in \mathbb{N},\text{ we have }
\lim\limits_{n\to\infty}\|T^n (\bz)\|=0$.     
\end{center}

Such $\delta$ is called an asymptotic expansivity constant of $T$.  
\end{definition} 

Motivated by Proposition \ref{positive_expansive}, we provide the following characterization of asymptotic expansivity:
\begin{proposition} 
\label{char_asymptotic} 
For a normed space $\ns$, let $T\in BL(\ns)$. Then, $T$ is asymptotically expansive if and only if either $\sup\limits_{n\in\mathbb{N}} \norm{T^n(\bx)}=\infty$ or $\lim\limits_{n\to\infty} \norm{T^n(\bx)}=0$ for all $\bx\in \mathcal{S}_{\mathcal{X}}$.     

Further, if $T$ is asymptotically expansive but not positively expansive, then there exists $\bx\in \mathcal{S}_{\mathcal{X}}$ such that $\lim\limits_{n\to\infty} \norm{T^n(\bx)}=0$.
\end{proposition} 

\begin{proof}
Let $T$ be asymptotically expansive with an asymptotic expansivity constant $\delta$. If $\bx\in \mathcal{S}_{\mathcal{X}}$ is such that $\sup\limits_{n\in\mathbb{N}} \norm{T^n(\bx)}=\infty$, then we are through. If not, then there exists real number $M>0$ such that $\norm{T^n(\bx)}\leq M$ for all $n\in\mathbb{N}$ and so,
$\norm{T^n\left(\frac{\delta}{M}\bx\right)}\leq \delta\mbox{ for all }n\in\mathbb{N}$. Then, by the definition of asymptotic expansivity of $T$, we obtain that $\lim\limits_{n\to\infty} \norm{T^n(\frac{\delta}{M}\bx)}=0$ and hence, $\lim\limits_{n\to\infty} \norm{T^n(\bx)}=0$.  
\medskip

Conversely, assume that either $\sup\limits_{n\in\mathbb{N}} \norm{T^n(\bx)}=\infty$ or $\lim\limits_{n\to\infty} \norm{T^n(\bx)}=0$ for all $\bx\in \mathcal{S}_\mathcal{X}$. Choose \(\delta>0\) arbitrarily. Let $\by\in \mathcal{X}\setminus \lbrace \ze\rbrace$ be such that $\norm{T^n(\by)}\leq \delta$ for all \(n\in\mathbb{N}\). Then, $\frac{\by}{\norm{\by}}\in \mathcal{S}_{\mathcal{X}}$ with $\norm{T^n\left(\frac{\by}{\norm{\by}}\right)}\leq \frac{\delta}{\norm{\by}}$ for all $n\in\mathbb{N}$. Since $\sup\limits_{n\in\mathbb{N}} \norm{T^n\left(\frac{\by}{\norm{\by}}\right)}$ cannot be \(\infty\), $\lim\limits_{n\to\infty} \norm{T^n\left(\frac{\by}{\norm{\by}}\right)}=0$, by our assumption. So, $\lim\limits_{n\to\infty} \norm{T^n(\by)}=0$. Therefore, $T$ must be asymptotically expansive.  
\medskip

Further, if $T$ is asymptotically expansive but not positively expansive, then by Proposition \ref{positive_expansive}, there exists $\bx\in\mathcal{S}_{\mathcal{X}}$ with $\sup\limits_{n\in\mathbb{N}}\norm{T^n(\bx)}\leq M$ for some positive real number $M$. Therefore, by the above chracterization of asymptotic expansivity, we have $\lim\limits_{n\to\infty} \norm{T^n(\bx)}=0$. 
\end{proof}

\begin{example}
\label{finitedim_example1}
Consider the non-invertible linear map $T:\mathbb{R}^3\rightarrow \mathbb{R}^3$ be given by
\[T\begin{pmatrix}
x_1
\\x_2
\\x_3 
\end{pmatrix}=\begin{pmatrix}
0 & 0 & 0\\0 & 2 & 0\\0 & 0 & 3    
\end{pmatrix}\begin{pmatrix}
x_1
\\x_2
\\x_3 
\end{pmatrix}.
\]
Clearly, \(0, 2\), and \(3\) are the eigenvalues of \(A\) and the corresponding eigenvectors are
\[\bv_1=(1,0,0)^T, \quad \bv_2=
(0,1,0)^T, \quad \bv_3=(0,0,1)^T.\]
Clearly, for \(\boldsymbol{v}=(x_1,x_2,x_3)^T\in\mathbb{R}^3\), we have \(\bv=x_1\bv_1+x_2\bv_2+x_3\bv_3.\) Further, \[T^n(\boldsymbol{v})=x_22^n\bv_2+x_33^n\bv_3 \mbox{ for all }n\geq 1.\] Therefore, $\norm{T^n(\boldsymbol{v})}=\sqrt{|x_2|^22^{2n}+|x_3|^23^{2n}}$ for all \(n\geq 1\). Hence, if either \(x_2\neq 0\) or \(x_3\neq 0\), then there is no \(\delta>0\) such that $\norm{T^n(\boldsymbol{v})}\leq \delta$ for all \(n\in\mathbb{N}\). Further, if \(x_2=x_3=0\), then
$T^n(\bv)=0$ for all \(n\geq 1\). Therefore, for a given \(\delta>0\), one may choose \(\bv\) such that
$T^n(\bv)\leq\delta$
for all \(n\in\mathbb{N}\). Clearly, for any such \(\bv\),
$\lim\limits_{n\to\infty}\norm{T^n(\bv)}=0$. 
Hence, \(T\) is asymptotically expansive. However, \(T\) is not positively expansive since \(\bv_1\neq \ze\) and $T^n(\bv_1)=0$
for all \(n\geq 1\).  
\end{example} 

\begin{example}
\label{finitedim_example2} 
The invertible linear map $T:\mathbb{R}^2\rightarrow \mathbb{R}^2$ be given by
\[T\begin{pmatrix}
x_1
\\x_2 
\end{pmatrix}=\begin{pmatrix}
\frac{1}{2} & 0\\0 & 3
\end{pmatrix}\begin{pmatrix}
x_1
\\x_2 
\end{pmatrix}
\]

is asymptotically expansive but not positively expansive. 
\end{example} 

\begin{example} 
\label{asymp-posi} 
Let us consider the incomplete normed space $\mathcal{X} = c_{00} \subset \ell^2(\mathbb{Z}^+)$, that is, the space of finitely supported sequences equipped with the $\ell^2$-norm: 
\[
\|\bx\| = \left(\sum_{k=1}^{\infty} |x_k|^2\right)^{1/2}
\] 
where \(\bx=(x_1,x_2,\ldots)\). Define the operator $T \in BL(\mathcal{X})$ by
\[
T(x_1, x_2, x_3, \dots) = \left( 2 x_1,\; \frac{1}{2}x_2,\; \frac{1}{3}x_3,\; \dots \right),
\]
or equivalently, $T(\be_1) = 2 \be_1$ and $T(\be_k) = \frac{1}{k} \be_k$ for $k \ge 2$.

\medskip
\noindent\textbf{Claim 1.} $T$ is asymptotically expansive on $\mathcal{X}$.

\medskip
Let $\bx \in \mathcal{X}$ satisfy $\|T^n (\bx)\| \le \delta$ for all $n \in \mathbb{N}$. The first coordinate $x_1$ must satisfy $2^n |x_1| \le \delta$ for all $n$, which forces $x_1 = 0$. Since $\bx$ has finite support, there exists $m \in \mathbb{N}$ such that $x_k = 0$ for all $k > m$, $\bx \in \text{span}\{\be_k : k \ge 2\}$, and
\[
\|T^n (\bx)\|^2 = \sum_{k=2}^m \left(\frac{1}{k}\right)^{2n} |x_k|^2 \le \left(\frac{1}{2}\right)^{2n} \sum_{k=2}^m |x_k|^2 = \left(\frac{1}{2}\right)^{2n} \|\bx\|^2.\] 

So, \[\|T^n(\bx)\|\leq \left(\frac{1}{2}\right)^n\|\bx\|\to 0 \mbox{ as } n\to \infty.\] 

Hence, $T$ is asymptotically expansive on $\mathcal{X}$. 

\medskip
\noindent\textbf{Claim 2.} $T$ is not positively expansive.

\medskip
Consider $\bx = \be_2 \in \mathcal{X}$. Then,
\[
\|T^n (\bx)\| = \left(\frac{1}{2}\right)^n \to 0 \quad \mbox{as } n \to \infty.
\]
Since $\bx \neq \ze$, by Proposition \ref{positive_expansive}, $T$ is not positively expansive.
\end{example}

The following example show that there are linear operators which are both positively expansive and hyperbolic. 

\begin{example} 
The invertible linear map $T:\mathbb{R}^2\rightarrow \mathbb{R}^2$ be given by
\[T\begin{pmatrix}
x_1
\\x_2 
\end{pmatrix}=\begin{pmatrix}
2 & 0\\0 & 3
\end{pmatrix}\begin{pmatrix}
x_1
\\x_2 
\end{pmatrix}
\]

is both positively expansive and hyperbolic. 
\end{example}

Further, the following example show that there are positively expansive linear operators which are not hyperbolic.  

\begin{example} 
Let $T:\ell^2(\mathbb{Z}^+)\rightarrow \ell^2(\mathbb{Z}^+)$ be given by $T(\be_k)=(1+\frac{1}{k})\be_k$ for all $k\in\mathbb{Z}^+$. Then, $T^n(\be_k)=(1+\frac{1}{k})^n\be_k$. Let $\bx=(x_1,x_2,\ldots)\neq \ze$ be such that $m$ is the smallest positive integer with $x_m\neq 0$. Then, $T^n(\bx)=\sum\limits_{k=m}^{\infty} (1+\frac{1}{k})^nx_k\be_k$ and hence, by the Parseval's identity  
\begin{equation*}
\|T^n(\bx)\|=\left(\sum\limits_{k=m}^{\infty} \left(1+\frac{1}{k}\right)^{2n} |x_k|^2\right)^{\frac{1}{2}}\geq |x_m|\left(1+\frac{1}{m}\right)^n\to \infty \mbox{ as } n\to \infty.     
\end{equation*}
This proves that $T$ is positively expansive. However,
\[\sigma(T)=\left\lbrace 1+\frac{1}{1},1+\frac{1}{2},\ldots\right\rbrace\cup \lbrace 1\rbrace.\]
So, $T$ is not hyperbolic.  
\end{example}

Furthermore, Example \ref{finitedim_example2} is hyperbolic but not positively expansive. 
\medskip

Now, we introduce strong asymptotic expansivity, a notion stronger than asymptotic expansivity. Further, we introduce super asymptotic expansivity, a notion stronger than strong asymptotic expansivity and prove that it is equivalent to hyperbolicity for linear operators on Hilbert spaces.

\begin{definition}
An operator $T\in BL(\mathcal{X})$ is said to be strong asymptotically expansive if there exists $\delta>0$ such that for every $\varepsilon > 0$ there exists $N=N(\varepsilon) \in \mathbb{Z}^+$ so that
\begin{center} 
whenever $\|T^n \bz\| \le \delta \quad \text{for all } n \in \mathbb{N}, \text{we have } 
\|T^n \bz\|\leq \varepsilon \quad \forall n \ge N$. 
\end{center} 

Such a $\delta$ is called a strong asymptotic expansivity constant of $T$.
\end{definition} 

\begin{proposition}\label{strongly_assymptotic}
For a normed space $\mathcal{X}$, let $T\in BL(\mathcal{X})$. Then, $T$ is strong asymptotically expansive with strong asymptotic expansivity constant $\delta$ if and only if, for all $\bx\in\mathcal{S}_\mathcal{X}$,\\
either (i) $\sup\limits_{n\in\mathbb{N}} \norm{T^n(\bx)}=\infty$,\\
or (ii) for every \(\varepsilon>0\), there exists \(N=N(\varepsilon)\in\mathbb{Z}^+\) such that $\norm{T^n(\bx)}\leq \frac{M}{\delta}\varepsilon$ for all \(n\geq N\), where \(M\) is an upper bound of the orbit of \(\bx\) (in particular, $M=\sup\limits_{n\in\mathbb{N}} \|T^n (\bx)\|$).\\ 

Further, if \(T\) is strong asymptotically expansive but not positively expansive, then there exists $\bx\in \mathcal{S}_{\mathcal{X}}$ satisfying the condition $(ii)$.  
\end{proposition}

\begin{proof}
Suppose that $T$ is strong asymptotically expansive with the strong asymptotically expansive constant $\delta$. If $\bx\in\mathcal{S}_\mathcal{X}$ is such that $\sup\limits_{n\in\mathbb{N}} \norm{T^n(\bx)}=\infty$, then there is nothing to prove. So, assume that $\|T^n(\bx)\|\leq M$ for some \(M>0\) and for all $n\in\mathbb{N}$. Then $\|T^n (\frac{\delta}{M} \bx)\|\leq \delta$ for all $n\in\mathbb{N}$. By strong asymptotic expansivity of $T$, for every $\varepsilon>0$ there exists $N=N(\varepsilon)\in\mathbb{Z}^+$ such that $\|T^n(\frac{\delta}{M}\bx)\|\leq \varepsilon$ for all $n\geq N$, that is, $\|T^n (\bx)\|\leq \frac{M}{\delta}\varepsilon$ for all $n\geq N$.  
\medskip

Conversely, let $\by \in \mathcal{X}\setminus \lbrace \ze\rbrace$ be such that $\norm{T^n(\by)}\leq \delta$ for all $n\in\mathbb{N}$. Then, $\frac{\by}{\norm{\by}}\in \mathcal{S}_{\mathcal{X}}$ with $\norm{T^n\left(\frac{\by}{\norm{\by}}\right)}\leq \frac{\delta}{\norm{\by}}$ for all $n\in\mathbb{N}$. Since $\sup\limits_{n\in\mathbb{N}} \norm{T^n\left(\frac{\by}{\norm{\by}}\right)}$ cannot be \(\infty\), for every $\varepsilon>0$ there exists $N=N(\varepsilon)\in\mathbb{Z}^+$ such that $\norm{T^n\left(\frac{\by}{\norm{\by}}\right)}\leq \frac{\delta}{\norm{\by}}\cdot\frac{\varepsilon}{\delta}$ for all $n\geq N$, that is, $\norm{T^n(\by)}\leq \varepsilon$ for all $n\geq N$.  
\end{proof}

It is easy to verify that strong asymptotic expansivity implies asymptotic expansivity. Moreover, the following example shows that there are asymptotically expansive linear operators which are not strong asymptotically expansive.

\begin{example}
\label{noncomplete_example1}
Let us consider the incomplete normed space $\mathcal{X} = c_{00}\subset \ell^2(\mathbb{N})$, that is, the space of finitely supported sequences equipped with the $\ell^2$-norm: 
\[
\|\bx\| = \left(\sum_{k=1}^{\infty} |x_k|^2\right)^{1/2}.
\]
Define the operator $T \in BL(\mathcal{X})$ by
\[
T(x_1, x_2, x_3, \dots) = \left(\left(1 - \frac{1}{1}\right)x_1,\; \left(1 - \frac{1}{2}\right)x_2,\; \left(1 - \frac{1}{3}\right)x_3,\; \dots\right),
\]
or equivalently, $T(\be_k)=\left(1 - \frac{1}{k}\right)\be_k$ for all $k \in \mathbb{N}$.

\medskip
\noindent\textbf{Claim 1.} $T$ is asymptotically expansive on $\mathcal{X}$.

\medskip
Let $\bz\in\mathcal{X}$ with $\|T^n(\bz)\| \le \delta$ for all $n \in \mathbb{N}$. Since $\bz$ has finite support, there exists $m \in \mathbb{N}$ such that $z_k = 0$ for all $k > m$. Then,
\[
T^n(\bz)=\sum_{k=1}^m \left(1 - \frac{1}{k}\right)^n z_k\be_k,
\quad\text{and hence}\quad
\|T^n(\bz)\|^2 = \sum_{k=1}^m \left(1 - \frac{1}{k}\right)^{2n} |z_k|^2.
\]
For each $k$, $\left(1 - \frac{1}{k}\right)^n \to 0$ as $n \to \infty$, so $\|T^n(\bz)\| \to 0$. Therefore, $T$ is asymptotically expansive on $\mathcal{X}$.

\medskip
\noindent\textbf{Claim 2.} $T$ is not strongly asymptotically expansive on $\mathcal{X}$.

\medskip
Fix $\delta>0$. For each $k \in \mathbb{Z}^+$, consider $\bz^{(k)} := \delta \be_k \in \mathcal{X}$. Then
\[
\|T^n(\bz^{(k)})\| = \delta \left(1 - \frac{1}{k}\right)^n \le \delta \quad \forall n\in\mathbb{N}.
\]

Suppose, for contradiction, there exists $N\in\mathbb{N}$ such that
\[
\|T^n(\bz)\| <\varepsilon=\frac{\delta}{2} \quad \forall n \ge N \quad \text{whenever } \|T^n(\bz)\| \le \delta \quad \forall n\in\mathbb{N}.
\]

Then in particular,
\[
\|T^N(\bz^{(k)})\|=\delta\left(1-\frac{1}{k}\right)^N<\frac{\delta}{2} \quad \forall k \in \mathbb{Z}^+.
\]

However, $(1 - 1/k)^N \to 1$ as $k \to \infty$, so for sufficiently large $k$,
\[
\|T^N(\bz^{(k)})\| > \frac{\delta}{2},
\]

a contradiction. Therefore, $T$ is not strongly asymptotically expansive.
\end{example}

We have strengthened the notion of asymptotic expansivity to strong asymptotic expansivity by demanding uniform convergence for the bounded orbits. For an easy understanding, we adopt the alternative characterizations (cf. Propositions \ref{char_asymptotic} and \ref{strongly_assymptotic}) of these notions. These characterizations direct us to apply another natural way to strengthen the notion of strong asymptotic expansivity by demanding uniform divergence for the unbounded orbits. 
\medskip

Further, we observe that hyperbolicity implies strong asymptotic expansivity and hence, asymptotic expansivity. However, the converse is not true in general (see Example \ref{asyexpnothyp} and Example \ref{strongnotsuper}). But the strongest variant among these three coincides with the notion of hyperbolicity. This is another motivation to adopt this variant called super asymptotic expansivity. 

\begin{definition}
An operator $T\in BL(\ns)$ is said to be super asymptotically expansive if for every $\varepsilon, M>0$ there exists $N=N(\varepsilon ,M)\in \mathbb{Z}^+$ so that, for all \(\bz\in \mathcal{S}_{\mathcal{X}}\),
\begin{equation*} 
\mbox{either }\|T^n(\bz)\| < \varepsilon\quad\mbox{or}\quad\|T^n(\bz)\| > M \quad \mbox{for all } n \ge N.
\end{equation*}
\end{definition}

The following example shows that there are strong asymptotically expansive linear operators which are not super asymptotically expansive. 

\begin{example} 
\label{strongnotsuper}  
Let $T:\ell^2(\mathbb{Z}^+)\rightarrow \ell^2(\mathbb{Z}^+)$ be given by \[T(\bx)=\left(\frac{1}{2}x_1,\frac{1}{3}x_2,\left(1+\frac{1}{3}\right)x_3,\left(1+\frac{1}{4}\right)x_4,\left(1+\frac{1}{5}\right)x_5,\ldots\right).\] 
Let $\bx=(x_1,x_2,\ldots)\in \ell^2(\mathbb{Z})$ be an arbitrary non-zero vector.   

If $m\geq 3$ is the smallest positive integer with $x_m\neq 0$, then $T^n(\bx)=\sum\limits_{k=m}^{\infty} (1+\frac{1}{k})^nx_k\be_k$ and hence, by the Parseval's identity

\begin{equation}
\label{3.1} 
\|T^n(\bx)\|=\left(\sum\limits_{k=m}^{\infty} \left(1+\frac{1}{k}\right)^{2n} |x_k|^2\right)^{\frac{1}{2}}\geq |x_m|\left(1+\frac{1}{m}\right)^n\to \infty \mbox{ as } n\to \infty.
\end{equation}    

Therefore, if $\|T^n(\bx)\|$ is bounded, then $x_k=0$ for all $k\geq 3$. Let $\varepsilon>0$ be given. If $x_1\neq 0$ and $x_2=0$, then there exists $N_1\in\mathbb{Z}^+$ such that $\|T^n(\bx)\|=\frac{1}{2^n}|x_1|<\varepsilon$ for all $n\geq N_1$. If $x_2\neq 0$ and $x_1=0$, then there exists $N_2\in\mathbb{Z}^+$ such that $\|T^n(\bx)\|=\frac{1}{3^n}|x_2|<\varepsilon$ for all $n\geq N_2$. If both $x_1\neq 0$ and $x_2\neq 0$, then $\|T^n(\bx)\|=\frac{1}{2^n}|x_1|+\frac{1}{3^n}|x_2|<\varepsilon$ for all $n\geq N=$max$\lbrace N_1,N_2\rbrace$. 

This proves that $T$ is strong asymptotically expansive. 

From equation \ref{3.1}, it is also clear that $T$ cannot be super asymptotically expansive.       
\end{example} 

\begin{remark}
The maps in Example \ref{noncomplete_example1} and Example \ref{asymp-posi} when considered on $\ell^2$ guarantee the existence of asymptotically expansive but not strongly asymptotically expansive, asymptotically expansive but not positively expansive linear operators on Banach spaces.        
\end{remark}

Later, in Theorem \ref{superequivalenthyperbolic}, we prove that the class of super asymptotic expansive linear operators and hyperbolic linear operators coincide on infinite dimensional complex Hilbert spaces. 
\medskip

In finite-dimensional setting, the class of asymptotically expansive and strong asymptotically expansive linear operators coincide with the super asymptotically expansive linear operators (or hyperbolic operators). On the other hand, Example \ref{finitedim_example1} and Example \ref{finitedim_example2} prove that the class of positively expansive linear operators is a different one. This fact is illustrated in Figure \ref{finite} and the same will be proved in Theorem \ref{equivalent_UAE}, Corollary \ref{strtosup}, Theorem \ref{hyperbolictosuper} and Corollary \ref{finitesuptohyp}. 

\begin{lemma}\label{lemma_subspace}
Let $\mathcal{X}$ be a normed space and $T \in BL(\mathcal{X})$. Then,
\[\mathcal{E}=\lbrace \bx\in \mathcal{X} : \|T^n(\bx)\|\to 0\mbox{ as }n\to \infty\rbrace\]
is a \(T\)-invariant subspace of $\mathcal{X}$.

Further, if \(T\) is surjective, then \(T(\mathcal{E})=\mathcal{E}\).
\end{lemma}

\begin{proof}
Clearly, \(\ze\in \mathcal{E}\). Further, if $\bx_1,\bx_2\in \mathcal{E}$ and \(\alpha\in\mathbb{K}\), then
\[\|T^n(\alpha \bx_1+\bx_2)\|\leq |\alpha|\|T^n(\bx_1)\|+\|T^n(\bx_2)\|\to 0\mbox{ as }n\to \infty.\] 
Therefore, \(\mathcal{E}\) is a subspace of \(\mathcal{X}\).

Now, for $\bx\in \mathcal{E}$,
\[\norm{T^n(T(\bx))}=\norm{T^{n+1}(\bx)}\to \ze\mbox{ as }n\to \infty.\]
Therefore, $T(\bx)\in \mathcal{E}$, that is,  \(T(\mathcal{E})\subseteq \mathcal{E}\), and hence \(\mathcal{E}\) is \(T\)-invariant.   
Further, assume that \(T\) is surjective. Let $\by\in\mathcal{E}$. Then, there exists $\bx\in\mathcal{X}$ such that $T(\bx)=\by$. Now, for \(n\geq 1\),
\(\norm{T^n(\bx)}=\norm{T^{n-1}(\by)}\), and hence
\[\norm{T^n(\bx)}\to 0\mbox{ as }n\to\infty.\]
This shows that $\bx\in \mathcal{E}$, that is, $\by\in T(\mathcal{E})$. Thus, $\mathcal{E}\subseteq T(\mathcal{E})$. Therefore, we conclude that $T(\mathcal{E})=\mathcal{E}$.
\end{proof}

\begin{theorem}\label{equivalent_UAE}
Let $\mathcal{B}$ be a finite-dimensional normed space and $T \in BL(\mathcal{B})$. Then, $T$ is asymptotically expansive if and only if it is strong asymptotically expansive.
\end{theorem} 

\begin{proof}
As noted earlier, strong asymptotic expansivity implies asymptotic expansivity. 
\medskip

Conversely, assume that $T$ is asymptotically expansive with asymptotic expansivity constant $\delta$. By Lemma \ref{lemma_subspace},
\[\mathcal{E}=\lbrace \bx\in \mathcal{B} : \|T^n(\bx)\|\to 0\mbox{ as }n\to \infty\rbrace\]
is a subspace of \(\mathcal{B}\). If $\mathcal{E}=\lbrace \ze\rbrace$, then the proof is trivial. Therefore, we assume that $\mathcal{E}\neq \lbrace \ze\rbrace$.
Since $\mathcal{B}$ is finite dimensional, \(\mathcal{E}\) is also so, and there is another subspace \(\mathcal{F}\) of $\mathcal{B}$, disjoint from $\mathcal{E}$, such that $\mathcal{B}=\mathcal{E}\oplus \mathcal{F}$. Clearly, for $\ze\neq \bz\in \mathcal{F}$, $\|T^n(\bz)\|\not\to 0$. Since $T$ is asymptotically expansive, $\sup\limits_{n\in\mathbb{Z}} \|T^n(\bz)\|=\infty$ by Proposition \ref{char_asymptotic}. 

Let $\bx\in \mathcal{B}$. Then, $\bx=\by+\bz$ for some $\by\in \mathcal{E}$ and $\bz\in \mathcal{F}$. If $\bz\neq \ze$, then
\[\|T^n(\bx)\|\geq |\|T^n(\bz)\|-\|T^n(\by)\||\mbox{, and hence, }\sup\limits_{n\in\mathbb{Z}} \|T^n(\bx)\|=\infty.\]
Now, assume that $\|T^n(\bx)\|\leq \delta$. Then, $\bz$ must be the zero vector. So, $\|T^n(\bx)\|=\|T^n(\by)\|\leq \delta$ for all $n\in\mathbb{N}$. 
Let \(\dim \mathcal{E}=k\) and $\lbrace \bb_1, \bb_2,\ldots,\bb_k\rbrace$ be a basis for $E$. Then, $\|T^n(\bb_i)\|\to 0$ for $i=1,2,\ldots,k$. Let $\by=\sum\limits_{i=1}^k \alpha_i \bb_i$ for some \(\alpha_i\in\fk\). Since $\mathcal{B}$ is finite dimensional, there exists constant $C>0$ such that $|\alpha_i|\leq C\|y\|\leq C\delta$ for $i=1,2,\ldots,k$. 

Now, let $\varepsilon>0$ be given. Since $\|T^n(\bb_i)\|\to 0$, there exists $N\in\mathbb{Z}^+$ such that $\|T^n(\bb_i)\|<\frac{\varepsilon}{C\delta k}$ for all $n\geq N$.  So,
\[\|T^n(\by)\|\leq \sum\limits_{i=1}^{k} |\alpha_i|\|T^n(\bb_i)\|\leq C\delta \sum\limits_{i=1}^{k} \|T^n(\bb_i)\|\leq \frac{C\delta\varepsilon}{C\delta}=\varepsilon\]
for all $n\geq N$. Thus, \(T\) is strong asymptotically expansive. This completes our proof.   
\end{proof} 

In contrast to Theorem \ref{equivalent_UAE}, whose proof relies only on finite-dimensional compactness, the passage from strong to super asymptotic expansivity requires promoting the unboundedness of an orbit to uniform convergence to infinity. This step appears to be closely tied to the spectral behaviour of the operator, and it may therefore be difficult to obtain a proof using only elementary functional-analytic arguments. We have two proofs: one involving the Cayley-Hamilton Theorem and the other following from the spectral characterization; we present only the latter in this article (see Corollary \ref{strtosup}). However, the spectral characterization requires complex spaces whereas the proof based on Cayley-Hamilton theorem is also applicable for real spaces.  
\medskip

Further, Theorem \ref{hyperbolictosuper} and Corollary \ref{infinitesuptohyp} implies that super asymptotic expansivity and hyperbolicity are equivalent.
\bigskip
\bigskip

The Venn Diagram in Figure \ref{finite} illustrate the complete scenario in finite dimension.

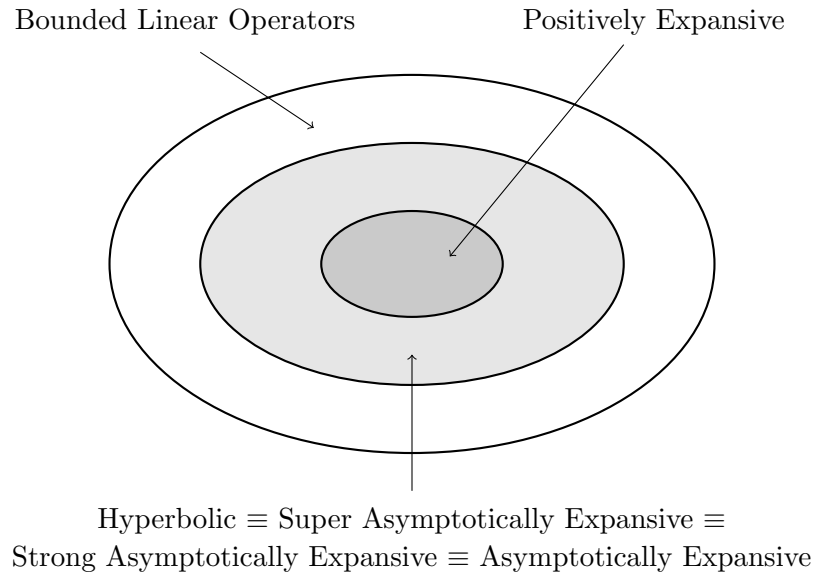
\begin{figure}[H]

\begin{center}
\begin{tikzpicture}[scale=1]

\draw[thick] (0,0) ellipse (4 and 2.5);
\node at (0,2.0) {}; 

\fill[gray!20] (0,0) ellipse (2.8 and 1.6);
\draw[thick] (0,0) ellipse (2.8 and 1.6);

\fill[gray!50, opacity=0.7] (0,0) ellipse (1.2 and 0.7);
\draw[thick] (0,0) ellipse (1.2 and 0.7);


\node at (3.2,3.2) {\small Positively Expansive};
\draw[->] (2.8,2.9) -- (0.5,0.1);
\node at (-3.0,3.2) {\small Bounded Linear Operators}; 
\draw[->] (-2.8,2.8) -- (-1.3,1.8);

\node at (0,-3.4) {\small Hyperbolic \(\equiv\) Super Asymptotically Expansive \(\equiv\)};
\node at (0,-3.9) {\small Strong Asymptotically Expansive \(\equiv\) Asymptotically Expansive};

\draw[->] (0,-3.0) -- (0,-1.2);

\node at (0,1.1) {}; 
\node at (0,0) {};  

\end{tikzpicture}
\end{center}

\caption{Venn Diagram of the Finite-Dimensional Setting}
\label{finite}
\end{figure} 

The Venn Diagram in Figure \ref{infinite} illustrate the complete scenario in infinite dimension. 

\begin{figure}[H]

\begin{center}
\begin{tikzpicture}[scale=1]

\draw[thick] (0,0) ellipse (4 and 2.5);
\node at (0,2.2) {};

\fill[gray!10] (0,0) ellipse (3.2 and 1.9);
\draw[thick] (0,0) ellipse (3.2 and 1.9);

\fill[gray!20] (0,0) ellipse (2.3 and 1.3);
\draw[thick] (0,0) ellipse (2.3 and 1.3);

\fill[gray!40, opacity=0.6] (-0.7,0) ellipse (1.2 and 0.7);
\draw[thick] (-0.7,0) ellipse (1.2 and 0.7);

\fill[gray!40, opacity=0.6] (0.7,0) ellipse (1.2 and 0.7);
\draw[thick] (0.7,0) ellipse (1.2 and 0.7);


\node at (-3.4,3.4) {\small Asymptotically Expansive};
\draw[->] (-3.0,3.1) -- (-1.2,1.4);

\node at (3.4,3.4) {\small Strong Asymptotically Expansive};
\draw[->] (3.0,3.1) -- (1.2,0.8);

\draw[->] (0.3,-4.5) -- (0.3,-2.2); 
\node at (0.5,-4.7) {\small Surjective Bounded Linear Operators};

\node at (-2.6,-3.4) {\small Super Asymptotically Expansive};
\node at (-2.5,-3.9) {\small (Hyperbolic)};
\draw[->] (-2.3,-3.0) -- (-1.2,-0.3);

\node at (2.6,-3.4) {\small Positively Expansive};
\draw[->] (2.3,-3.0) -- (1.2,-0.3);


\node at (0,1.6) {};
\node at (0,0.9) {};
\node at (-1.4,0) {};
\node at (1.4,0) {};
\end{tikzpicture}
\end{center}

\caption{Venn Diagram of the Infinite-Dimensional Setting} 
\label{infinite}
\end{figure}
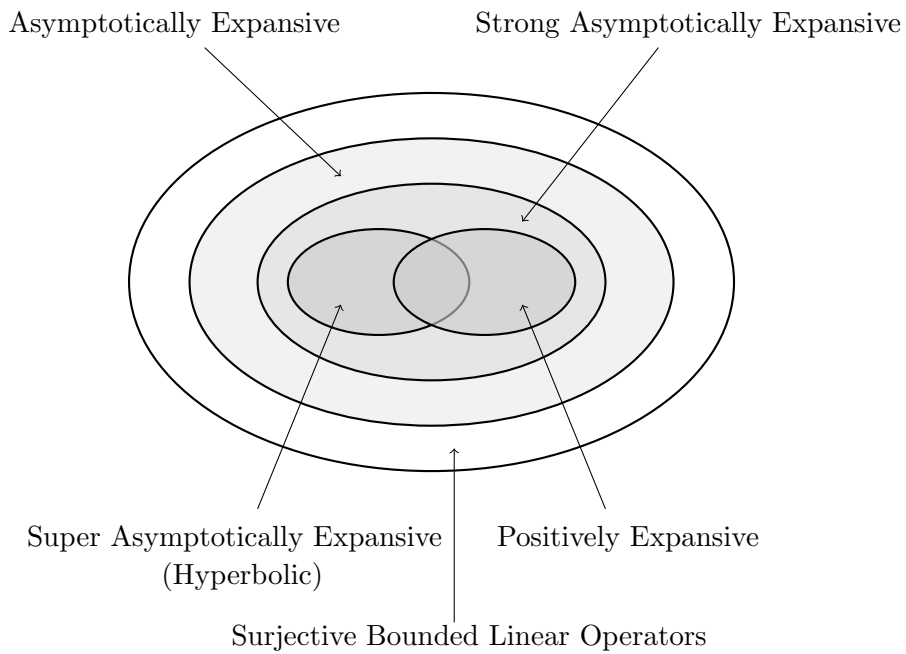

\section{Spectral Characterization}
\label{SpecChar}

\begin{proposition}\label{spec1}
Let $\mathcal{B}$ be a Banach space and let $T\in BL(\mathcal{B})$. Then, the following conditions are equivalent: 
\begin{itemize}
\item [(i)] \(\sigma(T)\subseteq \mathbb{D}\) 
\item [(ii)] There is \(c\in (0,1)\) and \(N_1\in\mathbb{Z}^+\) such that
$\|T^n(\bx)\|\leq c^n\|\bx\|\mbox{ for }\bx\in \mathcal{B}\mbox{ and }n\geq N_1.$
\end{itemize}
\end{proposition} 

\begin{proof}
Since $\sigma(T)\subseteq \mathbb{D}$, $0<r(T)<c<1$. So, there exists $N_1\in\mathbb{Z}^+$ such that $\norm{T^n}^{\frac{1}{n}}<c$ for all $n\geq N_1$ i.e., $\norm{T^n}<c^n$ for all $n\geq N_1$. 

Now, for $\bx\in \mathcal{B}$, we have 
\[\norm{T^n(\bx)}\leq \norm{T^n}\norm{\bx} \mbox{ for all } n\geq N_1 \]
which implies $\norm{T^n(\bx)}\leq c^n\norm{\bx}$ for all $n\geq N_1$. 
\smallskip

Conversely, assume that there exists $c\in (0,1)$ and $N_1\in \mathbb{Z}^+$ such that $\|T^n(\bx)\|\leq c^n\|\bx\|$ for all $\bx\in \mathcal{B}$ and $n\geq N_1$. Then, $\|T^n\|\leq c^n$ for $n\geq N_1$. Therefore, $0<r(T)\leq c<1$ and hence, $\sigma(T)\subseteq\mathbb{D}$.  
\end{proof}

\begin{proposition}\label{spec2}
If \(\sigma(T)\subset\mathbb{C}\setminus\overline{\mathbb{D}}\), then there exists \(d>1\) and \(N_2\in\mathbb{Z}^+\) such that
\[\|T^n(\bx)\|\geq d^n\|\bx\|\mbox{ for }\bx\in \mathcal{B}\mbox{ and }n\geq N_2.\] 
\end{proposition} 

\begin{proof}
Since $\sigma(T)=\lbrace \lambda\in \mathbb{C}\mid T-\lambda I$ is not invertible$\rbrace$ and $\sigma(T)\subset \mathbb{C}\setminus \overline{\mathbb{D}}$, therefore, $T-0I=T$ is invertible, i.e., $T^{-1}$ exists. By using the spectral mapping theorem (see \cite[Theorem B.6]{EP}), we obtain $\sigma(T^{-1})=\lbrace\frac{1}{\lambda}:\lambda\in \sigma(T)\rbrace$. So, $\sigma(T^{-1})\subset \mathbb{D}$ and hence, $r(T^{-1})<k<1$. So, there exists $N_2\in\mathbb{Z}^+$ such that $\|(T^{-1})^n\|^{\frac{1}{n}}<k$ for all $n\geq N_2$ i.e., $\|(T^{-1})^n\|<k^n$ for all $n\geq N_2$.  
\smallskip

Now, for $\bx \in \mathcal{B}$, we have 
\[\|\bx\|=\|(T^{-1})^n T^n (\bx)\|\leq \|(T^{-1})^n\|\|T^n (\bx)\| \mbox{ for all } n\geq N_2\] 
which implies $\|\bx\|\leq k^n\|T^n(\bx)\|$ for all $n\geq N_2$, that is, $\|T^n (\bx)\|\geq \frac{1}{k^n}\|\bx\|$ for all $n\geq N_2$. Choosing $d=\frac{1}{k}>1$, we obtain 
\[\|T^n(\bx)\|\geq d^n\|\bx\| \mbox{ for all } n\geq N_2.\] 

Hence, the result. 
\end{proof} 

The following lemma will be used to prove that the converse of Proposition \ref{spec2} holds only when the operator is surjective.  

\begin{lemma}
\label{inverse} 
Let $\mathcal{B}$ be a Banach space, $T\in BL(\mathcal{B})$ and $d>0$. If $\|T(\bx)\|\geq d\|\bx\|$ for all $\bx\in \mathcal{B}$, then $T^{-1}:Range(T)\to X$ is bounded with $\|T^{-1}\|\leq \frac{1}{d}$. 
\end{lemma} 

\begin{proof}
Assume that $T(\bx)=\ze$. Since $d\norm{\bx}\leq \norm{T(\bx)}=0$, we have $\norm{\bx}=0$ and hence, $\bx=\ze$. So, $T$ is injective. 

Now, take $T(\bx_k)\in Range(T)$ such that $T(\bx_k)\to \by$ in $\mathcal{B}$. We claim that $\by\in Range(T)$. 

Since $\|\bx_n-\bx_m\|\leq \frac{1}{d}\|T(\bx_n)-T(\bx_m)\|$, we have that $(\bx_k)$ is a Cauchy sequence. So, there exists $\bx\in \mathcal{B}$ such that $\bx_k\to \bx$. By continuity of $T$, $T(\bx)=\by$ i.e., $\by\in Range(T)$.  

This proves that $Range(T)$ is closed and $T:X\rightarrow Range(T)$ is bijective.  

Now, for $\by\in Range(T)$ with $\by=T(\bx)$, we have $\|T(\bx)\|\geq d\|\bx\|$ $\Rightarrow$ $\|T^{-1}(\by)\|\leq \frac{1}{d}\|\by\|$. Hence, $\|T^{-1}\|\leq \frac{1}{d}$.      
\end{proof} 

\begin{proposition}
\label{bigger} 
Let $\mathcal{B}$ be a Banach space and let $T\in BL(\mathcal{B})$ be surjective. If there exists $d>1$ and $N_2\in\mathbb{Z}^+$ such that $\|T^n(\bx)\|\geq d^n\|\bx\|$ for all $\bx \in \mathcal{B}$ and $n\geq N_2$, then $\sigma(T)\subseteq \mathbb{C}\setminus \overline{\mathbb{D}}$.   
\end{proposition} 

\begin{proof}
It is clear from the hypothesis that for each $n\geq N_2$, the operator $T^n$ is bounded below and hence, is injective. Further, since $T$ is surjective, $T^n$ is bijective and hence, by Lemma \ref{inverse}, $(T^n)^{-1}:\mathcal{B}\rightarrow \mathcal{B}$ is bounded with $\|(T^n)^{-1}\|\leq \frac{1}{d^n}$ i.e., $\|(T^{-1})^n\|\leq \frac{1}{d^n}$. This implies that $r(T^{-1})\leq \frac{1}{d}$. Since $d>1$, we have $\sigma(T^{-1})\subseteq \mathbb{D}\setminus \lbrace \ze\rbrace$. Therefore, $\sigma(T)\subseteq \mathbb{C}\setminus \overline{\mathbb{D}}$.   
\end{proof}

The following example shows that one cannot drop surjectivity from Proposition \ref{bigger}.  

\begin{example} 
\label{surjnecessary}
Let $\omega=(\omega_n)_{n\in\mathbb{Z}^+}$ with $\omega_n=2$ for all $n\geq 1$. Let $S_{\omega}:\ell^2(\mathbb{Z^+})\to \ell^2(\mathbb{Z^+})$ be given by $S_{\omega}\be_k=\omega_n\be_{k+1}$. Then, $\|S_{\omega}^n(\bx)\|\geq 2^n\|\bx\|$ for all $\bx\in \ell^2(\mathbb{Z^+})$ and $n\in\mathbb{Z}^+$. Here, $S_{\omega}$ is not surjective and in this case, $\sigma(S_{\omega})=\lbrace z : |z|\leq 2\rbrace$. So, $S_{\omega}$ is not hyperbolic.  
\end{example} 

We now give an example which shows that asymptotic expansivity does not imply hyperbolicity. 

\begin{example}
\label{asyexpnothyp} 
Let $S_{\omega}:\ell^2(\mathbb{Z^+})\to \ell^2(\mathbb{Z^+})$ be the weighted forward shift operator given by $S_{\omega}\be_k=\omega_n\be_{k+1}$, where $\omega=(\omega_j)_{j\in\mathbb{Z^+}}$.

If we choose $\omega_j=1-\dfrac{1}{j+1}=\dfrac{j}{j+1}$, we obtain 
\[S_{\omega}^n(\be_k)=\left(\prod\limits_{j=k}^{n+k-1}\omega_j\right)\be_{k+n}=\left(\prod\limits_{j=k}^{n+k-1} \frac{j}{j+1}\right)\be_{k+n}=\left(\frac{k}{n+k}\right)\be_{k+n}.\]
Therefore, for each \(k\),
\begin{equation}
\label{A}
\|S_{\omega}^n\be_k\|=\frac{k}{n+k}\to 0 \mbox{ as } n\to \infty.
\end{equation}  
\smallskip

Now, suppose that $\bx=\sum\limits_{k=1}^{\infty} \alpha_k\be_k$. Since $S_{\omega}$ is bounded, we obtain 
\[S_{\omega}^n(\bx)=S_{\omega}^n\left(\sum\limits_{k=1}^{\infty} \alpha_k\be_k\right)=\sum\limits_{k=1}^{\infty} \alpha_kS_{\omega}^n(\be_k)=\sum\limits_{k=1}^{\infty} \left(\frac{k}{k+n}\right)\alpha_k\be_{k+n}.\] 
This implies 
\begin{equation}
\label{B} 
\norm{S_{\omega}^n(\bx)}^2=\sum\limits_{k=1}^{\infty} \left(\frac{k}{k+n}\right)^2|x_k|^2<\sum\limits_{k=1}^{\infty}|x_k|^2<\infty. 
\end{equation} 
By \ref{A} and \ref{B}, and using the dominated convergence theorem, we obtain 
\[\norm{S_{\omega}^n(\bx)}\to 0 \mbox{ as } n\to \infty.\] 
Therefore, $S_{\omega}$ is asymptotically expansive.  
\smallskip

On the other hand, the spectral radius of \(S_{\omega}\) is given by
\[r(S_{\omega})=\limsup\limits_{n\to\infty} \left(\prod\limits_{j=1}^n \omega_j\right)^{\frac{1}{n}}=\lim\limits_{n\to \infty} \left(\frac{1}{n+1}\right)^{\frac{1}{n}}=1,\]
which shows that $S_{\omega}$ is not hyperbolic.   
\end{example}

Interestingly, the following discussion shows that a stronger class, namely the class of super asymptotically expansive operators coincide with the class of hyperbolic operators.    

\begin{theorem}
\label{hyptosup}
Let $\mathcal{B}$ be a complex Banach space and let $T\in BL(\mathcal{\mathcal{B}})$. If $T$ is hyperbolic, then there exist $0<c<1<d$ and $N\in\mathbb{Z}^+$ such that $\|T^n(\bx)\|\leq c^n\|\bx\|$ or $\|T^n(\bx)\|\geq d^n\|\bx\|$ for all $\bx\in \mathcal{B}$ and $n\geq N$.   
\end{theorem} 

\begin{proof} 
If $T$ is hyperbolic, then $\sigma(T)\subseteq \mathbb{T}^c$. Therefore, the spectrum can be written as
\[\sigma(T)=\sigma_1\sqcup \sigma_2,\]
where \(\sigma_1=\sigma(T)\cap\overline{\mathbb{D}}^c\) and \(\sigma_2=\sigma(T)\cap\mathbb{D}\). Hence, by Proposition \ref{RDT}, there are \(T\)-invariant subspaces \(\mathcal{B}_1\) and \(\mathcal{B}_2\) such that
\[\mathcal{H}=\mathcal{B}_1\oplus \mathcal{B}_2\mbox{ and } \sigma(T|_{\mathcal{B}_1})=\sigma_1, \sigma(T|_{\mathcal{B}_2})=\sigma_2.\] 

By Proposition \ref{spec1} and \ref{spec2}, there exist $0<c<1<d$ and $N\in\mathbb{Z}^+$ such that

\[\|T^n(\bx)\|\leq c^n\|\bx\| \mbox{ or } \|T^n(\bx)\|\geq d^n\|\bx\| \mbox{ for all } \bx\in \mathcal{H} \mbox{ and } n\geq N.\] 
Hence, the result.  
\end{proof}

\begin{corollary} 
\label{hyperbolictosuper} 
Let $\mathcal{B}$ be a complex Banach space and let $T\in BL(\mathcal{\mathcal{B}})$. Then, \(T\) is hyperbolic implies $T$ is super asymptotically expansive implies $T$ is strong asymptotically expansive implies $T$ is asymptotically expansive.  
\end{corollary}  

\begin{theorem}\label{converse_ps}
Let $\mathcal{B}$ be a finite dimensional complex Banach space and $T\in BL(\mathcal{B})$. If \(\sigma_p(T)\cap\mathbb{T}\neq\emptyset\), then \(T\) cannot be asymptotically expansive and hence cannot be strong asymptotically expansive and hence cannot be super asymptotically expansive.
\end{theorem}

\begin{proof}
Let \(\lambda\in\sigma_p(T)\). Therefore, there is \(\bx\neq\ze\) such that
\[T^n(\bx)=\lambda^n\bx\mbox{ for all }n.\]
Thus,
\[\norm{T^n(\bx)}=\norm{\bx}\mbox{ for all }n.\]
Therefore, \(\left\{\norm{T^n(\bx)}\right\}\) is bounded but it does not converge to \(0\). Hence, \(T\) cannot be asymptotically expansive by Proposition \ref{char_asymptotic}. Therefore, it cannot be strong asymptotically expansive.  
\end{proof} 

Thus, \(\sigma_p(T)\cap\mathbb{T}=\emptyset\) is a necessary condition for asymptotic expansivity, strong asymptotic expansivity and super asymptotic expansivity, but it is not sufficient as shown by the following example. 

\begin{example}
Consider the forward shift
\[S:\ell^2(\mathbb{Z}^+)\to\ell^2(\mathbb{Z}^+)\]
\[S(x_1,x_2,\cdots)=(0,x_1,x_2,\cdots)\] It is well known that \(S\) has no point spectrum, that is, \(\sigma_p(T)\cap\mathbb{T}=\emptyset\). However, since \(S\) is an isometry, it is not asymptotically expansive and hence, not strong asymptotically expansive and hence, not super asymptotically expansive.
\end{example}

\begin{example}
\label{nochar} 
Consider the forward shift 
\[S:\ell^2(\mathbb{Z}^+)\to\ell^2(\mathbb{Z}^+)\] 
\[ S(x_1,x_2,\cdots)=(0,x_1,x_2,\cdots)\] 

Then, $\sigma_p(S)=\emptyset, \sigma_{ap}(S)=\mathbb{T},\sigma(S)=\overline{\mathbb{D}}$, and \(S\) is not asymptotically expansive.     
\medskip

On the other hand, let $\omega=(\omega_n)_{n\in\mathbb{Z}^+}$, where $\omega_n=1-\frac{1}{n+1}$ for $n\in\mathbb{Z}^+$ and consider the weighted forward shift 
\[S_w:\ell^2(\mathbb{Z}^+)\to\ell^2(\mathbb{Z}^+)\]
\[ S_{\omega}\be_k=\omega_n\be_{k+1}\] 

Then, $\sigma_p(S_{\omega})=\emptyset, \sigma_{ap}(S_{\omega})=\mathbb{T},\sigma(S_{\omega})=\overline{\mathbb{D}}$,
and \(S_{\omega}\) is asymptotically expansive.  
\medskip

Therefore, no spectral characterization is possible for asymptotically expansive maps using the point, approximate point and whole spectrum of the map. It might be an interesting question to find a spectral characterization of such maps using some suitable subset of the spectrum.  
\end{example}

\begin{corollary} 
\label{strtosup}
Let $\mathcal{B}$ be a finite dimensional complex Banach space and $T\in BL(\mathcal{B})$. If $T$ is strong asymptotically expansive, then it is super asymptotically expansive.    
\end{corollary} 

\begin{proof}
In the finite-dimension, $\sigma(T)=\sigma_p(T)$. So, by Theorem \ref{converse_ps} we get that strong asymptotic expansivity implies hyperbolicity which implies super asymptotic expansivity by Theorem \ref{hyperbolictosuper}. 
\end{proof} 

\begin{corollary}
\label{finitesuptohyp}
Let $\mathcal{B}$ be a finite dimensional complex Banach space and $T\in BL(\mathcal{B})$. If $T$ is super asymptotically expansive, then it is hyperbolic.   
\end{corollary}

The following Lemma \ref{Spec_T11} will be used to prove Theorem \ref{infinitesuptohyp} which together with Example \ref{surjnecessary} show that the converse of Theorem \ref{hyptosup} holds only for a surjective operator on Hilbert spaces.  

\begin{lemma}\label{Spec_T11}
Let $\mathcal{H}$ be a complex Hilbert space, $T\in BL(\mathcal{H})$, and let \(\mathcal{H}_1\) be a closed \(T\)-invariant subspace of \(\mathcal{H}\). Then, with respect to the decomposition \(\mathcal{H}=\mathcal{H}_1\oplus\mathcal{H}_1^\perp\), \(T\) has the block upper-triangular representation
\[
T=
\begin{pmatrix}
T_{11} & T_{12}\\
0 & T_{22}
\end{pmatrix}.
\]
Further, the operator \(T_{22}\) is unitarily equivalent to the operator
\[
\widetilde{T}:\mathcal{H}/\mathcal{H}_1\to\mathcal{H}/\mathcal{H}_1
\quad\mbox{defined by}\quad
\widetilde{T}([\bx])=[T(\bx)],
\]
and consequently,
\[
\sigma(T)=\sigma(T_{11})\cup\sigma(\widetilde{T}).
\]
\end{lemma}

\begin{proof}
By the Projection Theorem, \(\mathcal{H}\) has the orthogonal decomposition $\mathcal{H}=\mathcal{H}_1\oplus\mathcal{H}_1^{\perp}$.
Let \(P:\mathcal{H}\to\mathcal{H}_1\) be the orthogonal projection onto
\(\mathcal{H}_1\). Since \(\mathcal{H}_1\) is \(T\)-invariant, \(T\) has the
following block upper-triangular representation with respect to the above decomposition of \(\mathcal{H}\):
\[
T=
\begin{pmatrix}
T_{11} & T_{12}\\
0 & T_{22}
\end{pmatrix},
\]
where \(T_{11}:\mathcal H_1\to\mathcal H_1\) is given by \(T_{11}=T|_{\mathcal H_1}\), \(T_{12}:\mathcal H_1^\perp\to\mathcal H_1\) is given by \(T_{12}=PT|_{\mathcal H_1^\perp}\), and \(T_{22}:\mathcal H_1^\perp\to\mathcal H_1^\perp\) is given by \(T_{22}=(I-P)T|_{\mathcal H_1^\perp}\). Moreover,
\[
\sigma(T)=\sigma(T_{11})\cup\sigma(T_{22}).
\]

Since \(\mathcal{H}_1\) is \(T\)-invariant, the operator
\(\widetilde{T}\) is well-defined. Now, define
\[
U:\mathcal{H}_1^\perp\to\mathcal{H}/\mathcal{H}_1\quad\mbox{by}\quad
U(\by)=[\by].
\]
It is routine to check that $U$ is bijective. Further, for
\(\by\in\mathcal{H}_1^\perp\),
\[
\|U(\by)\|=\inf_{\bh\in\mathcal H_1}\|\by+\bh\|=\|\by\|,
\]
since \(\by\perp\mathcal{H}_1\). Thus, \(U\) is unitary.

For \(\by\in\mathcal{H}_1^\perp\), we have \(T(\by)=T_{12}(\by)+T_{22}(\by)\), with \(T_{12}(\by)\in\mathcal{H}_1\), and hence
\[\widetilde{T}U(\by)=[T(\by)]=[T_{12}(\by)+T_{22}(\by)]=[T_{22}(\by)]=UT_{22}(\by).\]
Therefore, \(\widetilde{T}U=UT_{22}\), and consequently, \(T_{22}\) and \(\widetilde{T}\) are unitarily equivalent. Thus, \(\sigma(T_{22})=\sigma(\widetilde{T})\). Hence, we conclude that
\[
\sigma(T)=\sigma(T_{11})\cup\sigma(\widetilde{T}).
\]
\end{proof}   

\begin{theorem} 
\label{infinitesuptohyp} 
Let $\mathcal{H}$ be a complex Hilbert space and let $T\in BL(\mathcal{\mathcal{H}})$ be surjective. If there exist $0<c<1<d$ and $N\in\mathbb{Z}^+$ such that $\|T^n(\bx)\|\leq c^n\|\bx\|$ or $\|T^n(\bx)\|\geq d^n\|\bx\|$ for all $\bx\in \mathcal{H}$ and $n\geq N$, then $T$ is hyperbolic. 
\end{theorem}

\begin{proof}
Assume that there exist $0<c<1<d$ and $N\in\mathbb{Z}^+$ such that 
\[\|T^n(\bx)\|\leq c^n\|\bx\| \mbox{ or } \|T^n(\bx)\|\geq d^n\|\bx\| \mbox{ for all } \bx\in \mathcal{H} \mbox{ and } n\geq N.\] 
Now, consider
\begin{align*}
&\mathcal{H}_1=\lbrace \bx\in \mathcal{H} : \norm{T^n(\bx)}\to \ze\mbox{ as }n\to \infty\rbrace,\\
\mbox{and}\quad &\mathcal{H}_2=\lbrace \bx\in\mathcal{H} : \|T^n(\bx)\|\leq c^n\|\bx\|\mbox{ for all }n\geq N\rbrace.
\end{align*}
Clearly, $\mathcal{H}_2\subseteq \mathcal{H}_1$. Further, if $\norm{ T^n(\bx)}\to \ze$ as $n\to \infty$, then it cannot happen that $\norm{T^n(\bx)}\geq d^n\norm{\bx}$ for all $n\geq N$ and hence, we have $\|T^n(\bx)\|\leq c^n\|\bx\|$ for all $n\geq N$. Therefore, $\mathcal{H}_1\subseteq \mathcal{H}_2$ and hence, $\mathcal{H}_1=\mathcal{H}_2$.
Since $\mathcal{H}_2$ is a countable intersection of the closed sets
\[\mathcal{E}_n=\lbrace \bx\in\mathcal{H} : \|T^n(\bx)\|\leq c^n\|\bx\|\rbrace,\]
it is closed. Therefore, by Lemma \ref{lemma_subspace}, \(\mathcal{H}_1\) is a closed \(T\)-invariant subspace of \(\mathcal{H}\), with \(T(\mathcal{H}_1)=\mathcal{H}_1\). Hence, by Lemma \ref{Spec_T11},
\[\sigma(T)=\sigma(T_{11})\cup \sigma(\widetilde{T}),\]
the notations are as stated in the lemma. We show that $\sigma(T_{11})\subseteq D$ and $\sigma(\widetilde{T})\subseteq \mathbb{C}\setminus \overline{D}$ to prove $T$ is hyperbolic. 

Since $T_{11}=T|_{\mathcal{H}_1}$, by Proposition \ref{spec1} and the given hypothesis, we have $\sigma(T_{11})\subseteq D$. 

We now prove that there exists $d>1$ and $N\in\mathbb{Z}^+$ such that $\norm{\widetilde{T}^n([\bx])}\geq d^n\norm{[\bx]}$, for all $n\geq N$. Then, it follows from Proposition \ref{bigger} that $\sigma(\widetilde{T})\subseteq \mathbb{C}\setminus\overline{D}$. Since \(T(\mathcal{H}_1)=\mathcal{H}_1\), we obtain
\[\widetilde{T}^n([\bx])=[T^n(\bx)]=T^n(\bx)+\mathcal{H}_1=\lbrace T^n(\bx)+T^n(\bh):\bh\in\mathcal{H}_1\rbrace.\]   
Therefore, 
\[\norm{\widetilde{T}^n([\bx])}=\inf\limits_{\bh\in\mathcal{H}_1} \norm{T^n(\bx)+T^n(\bh)}=\inf\limits_{\bh\in \mathcal{H}_1} \norm{T^n(\bx+\bh)}.\]  
If $\bx\in \mathcal{H}_1$, $\norm{\widetilde{T}^n([\bx])}=0\geq d^n\norm{[\bx]}$. If $\bx\notin \mathcal{H}_1$, then $\bx+\bh\notin \mathcal{H}_1$ for any $\bh\in \mathcal{H}_1$. By the hypothesis, there exists $d>1$ and $N\in\mathbb{Z}^+$ such that $\norm{T^n(\bx+\bh)}\geq d^n\norm{\bx+\bh}$ for all \(\bh\in\mathcal{H}_1\). Taking infimum over $\bh\in\mathcal{H}_1$, we obtain 
\begin{align*}
\inf\limits_{\bh\in\mathcal{H}_1}\norm{T^n(\bx+\bh)} & \geq d^n\inf\limits_{\bh\in\mathcal{H}_1}\norm{\bx+\bh},\\
\mbox{that is,}\quad \norm{\widetilde{T}^n([\bx])} & \geq d^n\norm{[\bx]}.
\end{align*}
This completes the proof.  
\end{proof} 

\begin{corollary}
\label{supertohyperbolic} 
Let $\mathcal{H}$ be a complex Hilbert space and let $T\in BL(\mathcal{\mathcal{H}})$ be surjective. If $T$ is super asymptotically expansive, then $T$ is hyperbolic.  
\end{corollary} 

Combining Corollary \ref{hyperbolictosuper} and Corollary \ref{supertohyperbolic} we obtain the following important result. 

\begin{theorem}
\label{superequivalenthyperbolic} 
Let $\mathcal{H}$ be a complex Hilbert space and let $T\in BL(\mathcal{\mathcal{H}})$ be surjective. Then, $T$ is super asymptotically expansive if and only if $T$ is hyperbolic. 
\end{theorem} 

By the stability of the semi-Fredholm index, it is well established that the class of hyperbolic operators fails to be dense in the operator norm topology. Specifically, sufficiently small norm perturbations cannot resolve a non-zero index on the unit circle, thereby yielding open sets that contain no hyperbolic operators. This behavior parallels the well-documented norm-closure obstructions associated with invertible operators (see \cite{C}). Consequently, investigations into the approximation of general bounded linear operators by hyperbolic ones must naturally proceed within the context of weaker topologies, such as the strong operator topology (SOT). 
\medskip

While the SOT-density of hyperbolic operators is widely recognized as a direct corollary of finite-rank approximation techniques, an explicit formalization of this result seems to be not appearing in the standard literature. Therefore, for the sake of completeness, we include a formal demonstration of this property below.
\medskip 

Here, we use Remark \ref{spectrum_fr} and Lemma \ref{finite-rank} to present the proof of SOT-density on infinite dimensional complex Banach space. The proof is similar if the space is finite dimensional.  

\begin{remark}
\label{spectrum_fr}
Let \(\mathcal{B}\) be a complex Banach space and \(F\in BL(\mathcal{B})\) is a finite-rank operator. Then,
\(\sigma(F)=\sigma_p(F)\)
is a finite subset of \(\cn\) containing $0$.     
\end{remark}

\begin{lemma}\label{finite-rank}
Let \(\mathcal{B}\) be a complex Banach space and \(T\in BL(\mathcal{B})\). Then, there exists a sequence \((F_n)_{n\in\mathbb{Z}^+}\) in $BL(\mathcal{B})$ of finite-rank operators that converges to \(T\) in SOT. 
\end{lemma}

\begin{proof}
Let \(\{\bb_1, \bb_2, \ldots\}\) is a Schauder basis. For \(n\in\zn^+\), let
\[\mathcal{B}_n=\spn (\{\bb_1,\bb_2,\ldots,\bb_n\}).\]
It is easy to see that \(\mathcal{B}_n\subset\mathcal{B}_{n+1}\) and \(\bigcup\limits_{n=1}^{\infty}\mathcal{B}_n=\mathcal{B}\). Further, for each \(n\), \(\mathcal{B}\) can be written as \(\mathcal{B}_n\oplus\mathcal{B}_n'\). Therefore, every \(\bx\in\mathcal{B}\) can be written as \(\bx=\bx_n+\bx_n'\) uniquely for some \(\bx_n\in\mathcal{B}_n\) and \(\bx_n'\in\mathcal{B}_n'\). Define the projections \(P_n:\mathcal{B}\to\mathcal{B}\) by \(P_n(\bx)=\bx_n\), that is, \(P_n|_{\mathcal{B}_n}=I_{\mathcal{B}_n}\). Therefore, it follows that
\[P_n(\bx)\to\bx\mbox{ as }n\to\infty.\]
In other words, \(P_n\to I\) in SOT.

Let us now define the operators \(F_n=TP_n\) for \(n\in\zn^+\). It is easy to see that each \(F_n\) is a finite-rank operator. Moreover, \(F_n\to T\) in SOT.
\end{proof}

\begin{theorem}
Let \(\mathcal{B}\) be a complex Banach space. Then, the set of all super asymptotically expansive operators on \(\mathcal{B}\) is SOT dense in \(BL(\mathcal{B})\).
\end{theorem}

\begin{proof}
Let \(T\in BL(\mathcal{B})\). We construct a sequence of super asymptotically expansive operators that converges to \(T\) in SOT. Define the operators \(P_n\) and \(F_n\) as in Lemma \ref{finite-rank}. Clearly, \(F_n-P_n\) is a finite-rank operator and, by Remark \ref{spectrum_fr},
\[\sigma(F_n-P_n)=\sigma_p(F_n-P_n)=\{\lambda_1,\ldots,\lambda_k\}\]
for some \(\lambda_i\in\cn\) and \(k\in\zn^+\). Let \((\varepsilon_n)\) be a sequence of real numbers such that \(\varepsilon_n\to 0\) as \(n\to\infty\). For each \(n\), define the operators
\[T_n=(1+\varepsilon_n)I+(F_n-P_n).\]
Clearly,
\[\sigma(T_n)=\{1+\varepsilon_n+\lambda_i:i=1,2,\ldots,k\}.\]
One may choose \(\varepsilon_n\) in such a way that \(\sigma(T_n)\cap\mathbb{T}=\emptyset\), that is, each \(T_n\) is hyperbolic. Further,
\[T_n(\bx)=(1+\varepsilon_n)I(\bx)+TP_n(\bx)-P_n(\bx)\to T(\bx)\quad \mbox{as }n\to\infty.\]
This completes the proof.  
\end{proof}

\section{Construction of Operators in Specific Classes}

In this section, we make use of the fundamental operations of \(BL(\mathcal{X})\) to obtain new operators in specific classes.

\begin{proposition}
\label{cT}
Let \(\mathcal{X}\) be a normed space, $T \in BL(\mathcal{X})$, and \(c\in\mathbb{T}\). Then, the following hold: 
\begin{itemize}
\item [(a)] $T$ is asymptotically expansive if and only if $cT$ is so. \item [(b)] $T$ is strong asymptotically expansive if and only if $cT$ is so. 
\item [(c)] $T$ is super asymptotically expansive if and only if $cT$ is so.   
\end{itemize}
\end{proposition}

\begin{proof}
Since \(|c|=1\), we obtain
\[\norm{(cT)^n(\bx)}=|c|^n\norm{T^n(\bx)}=\norm{T^n(\bx)}\]
for all \(\bx\in X\) and \(n\in\mathbb{N}\). Therefore, the result follows from the definitions.
\end{proof}

\begin{proposition}\label{spectrum}
Let \(\mathcal{B}\) be a complex Banach space, \(T,P\in BL(\mathcal{B})\) and \(P\) be invertible. Then,
\begin{align*}
&\sigma(P^{-1}TP)=\sigma(T)\quad\mbox{and}\quad\sigma_p(P^{-1}TP)=\sigma_p(T);\\
&\sigma(cT)=c\sigma(T)\quad\mbox{and}\quad\sigma_p(cT)=c\sigma_p(T)\quad\mbox{for }c\in\cn.
\end{align*}
\end{proposition}

For a complex Banach space \(\mathcal{B}\) and \(T\in BL(\mathcal{B)}\), define
\begin{align*}
& \tau(T)=\{t\in\cn : |t\lambda|=1\mbox{ for some }\lambda\in\sigma(T)\}\\
\mbox{and}\quad & \tau_p(T)=\{t\in\cn : |t\lambda|=1\mbox{ for some }\lambda\in\sigma_p(T)\}.
\end{align*}

\begin{proposition}\label{super_ct}
Let \(\mathcal{B}\) be a complex Banach space, \(T\in BL(\mathcal{B})\), and \(c\in\cn\). Then, \(cT\) is super asymptotically expansive if and only if \(c\notin\tau(T)\).
\end{proposition}

\begin{proof}
By Proposition \ref{spectrum}, \(\sigma(cT)=c\sigma(T)\). Therefore, \(\sigma(cT)\cap\mathbb{T}=\emptyset\) iff
\[|c\lambda|\neq 1\quad\mbox{ for all }\lambda\in\sigma(T)\]
if and only if \(c\notin\tau(T)\). Hence, by Proposition \ref{cT}, \(cT\) is super asymptotically expansive if and only if \(c\notin\tau(T)\).  
\end{proof}

\begin{remark}
In the above proposition, it is interesting to note that \(T\) itself need not be super asymptotically expansive. For example, consider the map
\[T:\cn^2\to\cn^2\quad\mbox{given by}\quad T(x_1,x_2)=\left(x_1, 2x_2\right).\]
Clearly, \(\sigma(T)=\left\{1,2\right\}\). Therefore, \(T\) is not super asymptotically expansive, but \(\frac{1}{4}T\) is.

This result might be very useful to construct fimilies of examples of super asymptotically expansive operators.

One may also note that super asymptotic expansivity of \(T\) does not guarantee the same of \(cT\) for every \(c\in\cn\).
\end{remark}
 
\begin{corollary}
\label{partialResult} 
Let \(\mathcal{B}\) be a complex Banach space and \(T\in BL(\mathcal{B})\). If \(c\notin \tau(T)\), then 
\begin{itemize}
\item [(a)] \(cT\) is strong asymptotically expansive, 
\item [(b)] \(cT\) asymptotically expansive.   
\end{itemize}
\end{corollary}

\begin{proposition}
Let \(\mathcal{B}\) be a complex Banach space and \(T\in BL(\mathcal{B})\). If \(c\in\tau_p(T)\), then \(cT\) cannot be asymptotically expansive, and hence, cannot be strong asymptotically expansive.
\end{proposition}
\begin{proof}
By Proposition \ref{spectrum}, \(\sigma_p(cT)=c\sigma_p(T)\). Since \(c\in\tau_p(T)\),
\[|c\lambda|=1\quad\mbox{ for some }\lambda\in\sigma_p(T).\]
Therefore, \(\sigma_p(cT)\cap\mathbb{T}\neq\emptyset\). Hence, by Theorem \ref{converse_ps}, \(cT\) cannot be asymptotically expansive.
\end{proof}

The following proposition is an adaptation of \cite[Lemma 3.14]{BDOP} for strong asymptotic expansivity. One can adopt similar argument for asymptotic expansivity and super asymptotic expansivity. 

\begin{proposition}
Let $\mathcal{X}$ and $\mathcal{Y}$ be normed spaces over \(\mathbb{K}\), and let $L:\mathcal{X}\to\mathcal{Y}$ be an invertible bounded linear map. Let
$T\in BL(\mathcal{X})$ and define $S=LTL^{-1}\in BL(\mathcal{Y})$. Then, $T$ is strong asymptotically expansive if and only if $S$ is strong asymptotically
expansive. 
\end{proposition} 

\begin{proof}
Assume that $T$ is strong asymptotically expansive. Let $\varepsilon>0$ be given. Let us denote
\(\alpha=\norm{L}\) and \(\beta=\norm{L^{-1}}\).
Then, we have the following:
\begin{equation}\label{relation_norms}
\frac{1}{\beta}\norm{\bx}\leq\|L(\bx)\|\leq \alpha\|\bx\|
\quad\mbox{and}\quad
\frac{1}{\alpha}\norm{\by}\leq\|L^{-1}(\by)\|\leq \beta\|\by\|
\end{equation}
for all $\bx\in\mathcal{X}$ and $\by\in\mathcal{Y}$.

Let $\by\in\mathcal{S}_{\mathcal{Y}}$ and set
\[
\bx=\frac{L^{-1}(\by)}{\|L^{-1}(\by)\|}\in\mathcal{S}_{\mathcal{X}}.
\]
Then,
\(
L^{-1}(\by)=\|L^{-1}(\by)\|\bx.
\)
Since $T$ is linear, for every $n\in\mathbb{N}$,
\begin{equation}\label{relation_ST}
\norm{S^n(\by)}
=\norm{L(T^n(L^{-1}(\by)))}
=\|L^{-1}(\by)\|\norm{L(T^n(\bx))}.
\end{equation}

Now, by the strong asymptotic expansivity of $T$, we have for all $\bx\in\mathcal{S}_\mathcal{X}$, 
\\
either (i) $\sup\limits_{n\in\mathbb{N}} \norm{T^n(\bx)}=\infty$,\\
or (ii) for every \(\varepsilon>0\) there exists \(N=N(\varepsilon)\in\mathbb{Z}^+\) such that $\norm{T^n(\bx)}\leq \frac{M_1}{\delta}\varepsilon$ for all \(n\geq N\), where $\|T^n (\bx)\| \leq M_1$ for all \(n\in\mathbb{N}\).\\

Consider first the case when $\sup\limits_{n\in\mathbb{N}} \norm{T^n(\bx)}=\infty$. Using \(\norm{\by}=1\) and \eqref{relation_norms}, we obtain from \eqref{relation_ST},
\[
\|S^n(\by)\|
\geq \frac{1}{\alpha\beta}\|T^n(\bx)\| \mbox{ and hence, } \sup\limits_{n\in\mathbb{N}} \norm{S^n(\by)}=\infty.
\]

Consider next the case when $\norm{T^n(\bx)}\leq \frac{M_1}{\delta}\varepsilon$. Using \(\norm{\by}=1\) and \eqref{relation_norms}, we obtain from \eqref{relation_ST},
\[
\|S^n(\by)\|
\leq \alpha\beta\norm{T^n(\bx)}.
\]
Then, \(M=\alpha\beta M_1\) is an upper bound of \((\|S^n(\by)\|)_{n\in\nn}\), and by the given inequality,
\[
\|S^n(\by)\|\leq\alpha\beta\frac{M_1}{\delta}\varepsilon=\frac{M}{\delta}\varepsilon.
\]
This shows that for all $\by\in\mathcal{S}_\mathcal{Y}$,\\
either (i) $\sup\limits_{n\in\mathbb{N}} \norm{S^n(\by)}=\infty$,\\
or (ii) for every \(\varepsilon>0\) there exists \(N=N(\varepsilon)\in\mathbb{Z}^+\) such that $\norm{S^n(\by)}\leq \frac{M}{\delta}\varepsilon$ for all \(n\geq N\), where $\|S^n (\bx)\| \leq M$ for all \(n\in\mathbb{N}\).

Hence, $S$ is strong asymptotically expansive by Proposition \ref{strongly_assymptotic}. Since $L:\mathcal{X}\rightarrow\mathcal{Y}$ is invertible, 
the converse follows similarly.  
\end{proof}

\begin{theorem}
Let \(\hs\) be a Hilbert space, \(T\in BL(\hs)\), and \(T\) be normal. Then, the following hold:
\begin{itemize}
\item [(a)] $T$ is asymptotically expansive if and only if $T^*$ is so. \item [(b)] $T$ is strong asymptotically expansive if and only if $T^*$ is so. 
\item [(c)] $T$ is super asymptotically expansive if and only if $T^*$ is so. 
\end{itemize}
\end{theorem}

\begin{proof}
Since \(T\) is normal, \(T^n\) is also so for all \(n\in\mathbb{N}\). Therefore, \(\norm{T^n(\bx)}=\norm{(T^*)^n(\bx)}\) for all \(\bx\in X\) and \(n\in\mathbb{N}\). Therefore, the implications follow from the definitions. 
\end{proof}

We now prove that asymptotic expansivity is preserved under taking positive integral power. It is easy to verify that the same result hold for strong asymptotic expansivity and super asymptotic expansivity.   

\begin{proposition}
Let \(\mathcal{X}\) be a normed space, and $T \in BL(\mathcal{X})$ be asymptotically expansive. Then $T^r$ is asymptotically expansive for every $r \in \mathbb{Z}^+$.
\end{proposition}

\begin{proof}
Let $r \in \mathbb{Z}^+$ be fixed and assume that $T$ is asymptotically expansive. We show that $T^r$ is also asymptotically expansive.

Let $\bx \in \mathcal{X}$ be such that the orbit under $T^r$ is bounded, that is,
\[
\sup_{m \in \mathbb{N}} \|(T^r)^m (\bx)\|=\sup_{m \in \mathbb{N}} \|T^{rm} (\bx)\| \leq M
\]
for some $M > 0$. We first claim that the full orbit $\{T^n (\bx)\}_{n \in \mathbb{N}}$ is bounded. Indeed, for each $n \in \mathbb{N}$, write
\[
n = rm + k, \quad \text{where } m \in \mathbb{N}, \; k \in \{0,1,\dots,r-1\}.
\]
Then,
\[
T^n (\bx) = T^{rm + k} (\bx) = T^k (T^{rm} (\bx)).
\]
Taking norms and using the boundedness of $T^k$, we obtain
\[
\|T^n (\bx)\| \le \|T^k\| \, \|T^{rm} (\bx)\|.
\]
Hence,
\[\|T^n (\bx)\| \le \left(\max_{0 \le k \le r-1} \|T^k\|\right)
\left(\sup_{m \in \mathbb{N}} \|T^{rm} (\bx)\|\right).\] 
Since each $T^k$ is bounded and there are only finitely many such powers, the constant
\(
C_r := \max\limits_{0 \le k \le r-1} \|T^k\|
\)
is finite. Therefore, $\sup_{n \in \mathbb{N}} \|T^n (\bx)\| \le C_r M < \infty$,
that is, the full orbit of $\bx$ under $T$ is bounded.

Since $T$ is asymptotically expansive, it follows that
\[
\|T^n (\bx)\| \to 0 \quad \mbox{as } n \to \infty.
\]
In particular, taking the subsequence $n = rm$, we obtain
\[
\|(T^r)^m (\bx)\|=\|T^{rm} (\bx)\| \to 0\quad\mbox{as }m\to\infty.
\]
Therefore, $T^r$ is asymptotically expansive by Proposition \ref{char_asymptotic}.
\end{proof}
\medskip

\textbf{\large Open Questions}
\medskip

In view of Example \ref{nochar}, the following question remain open for now.  
\medskip

Open Question 1: Is there any subset of the spectrum which characterizes asymptotic expansivity? 
\medskip

Even though Proposition \ref{super_ct} provides a complete characterization of scalars \(c\) for which \(cT\) is super asymptotically expansive, we could not provide such a characterization for the other variants. Based on the partial result Corollary \ref{partialResult}, the following questions remain open for now. 
\medskip

Open Question 2: Find subsets $a(T)\subset s(T)\subset \tau(T)$ such that 
\begin{itemize}
\item [(a)] $cT$ is strong asymptotically expansive if and only if $c\notin s(T)$. 
\item [(b)] $cT$ is asymptotically expansive if and only if $c\notin a(T)$. 
\end{itemize} 

Open Question 3: Is the Hilbert space assumption in Theorem \ref{infinitesuptohyp} necessary?

\end{document}